\documentclass[12pt]{amsart}

\usepackage[matrix,arrow,curve,frame]{xy}
\usepackage{amsmath,amsthm,amssymb,enumerate}
\usepackage{latexsym}
\usepackage{amscd}
\usepackage[colorlinks=false]{hyperref}
\usepackage{euscript}

\newtheorem{thm}[subsection]{Theorem}

\newtheorem{cor}[subsection]{Corollary}
\newtheorem{lemma}[subsection]{Lemma}

\theoremstyle{definition}

\numberwithin{equation}{section}

\def\lie#1{{\mathfrak #1}}

\def\phi{{\varphi}}

\def\cI{{\mathcal I}}

\def\cV{{\mathcal V}}
\def\cW{{\mathcal W}}

\def\gg{{\mathfrak g}}

\def\gl{{\mathfrak l}}

\def\gs{{\mathfrak s}}

\newfont{\german}{eufm10}

\begin{document}
\pagestyle{plain}

\title
{The $\mathbb{Z}_{2}$-orbifold of the universal affine vertex algebra}

\author{Masoumah Al-Ali}
\email{zlutf95@hotmail.com}

\begin{abstract} Let $\gg$ be a simple, finite-dimensional complex Lie algebra, and let $V^k(\gg)$ denote the universal affine vertex algebra associated to $\gg$ at level $k$. The Cartan involution on $\gg$ lifts to an involution on $V^k(\gg)$, and we denote by $V^k(\gg)^{\mathbb{Z}_2}$ the orbifold, or fixed-point subalgebra, under this involution. Our main result is an explicit minimal strong finite generating set for $V^k(\gg)^{\mathbb{Z}_2}$ for generic values of $k$. In the case $\gg = \gs\gl_2$, we also determine the set of nongeneric values of $k$, where this set does not work. \end{abstract}

\maketitle
\section{Introduction}
 Starting with a vertex algebra $\mathcal{V}$ and a group $G$ of automorphisms of $\mathcal{V}$, the invariant subalgebra $\mathcal{V}^G$ is called a $G$-{\it orbifold} of $\mathcal{V}$. Many interesting vertex algebras can be constructed either as orbifolds or as extensions of orbifolds. A remarkable example is the Moonshine vertex algebra $V^{\natural}$, which is an extension of the $\mathbb{Z}_2$-orbifold of the lattice vertex algebra associated to the Leech lattice \cite{B,FLM}. A substantial literature has evolved on the structure and representation theory of orbifolds under finite group actions including \cite{DVVV,DHVW,DM,DLMI,DLMII,DRX}. It is widely believed that nice properties of $\mathcal{V}$ such as $C_2$-cofiniteness and rationality will be inherited by $\mathcal{V}^G$ when $G$ is finite. In \cite{M}, Miyamoto proved the $C_2$-cofiniteness of $\mathcal{V}^G$ when $G$ is cyclic. Also, he recently established the rationality with Carnahan in \cite{CM}.

Many vertex algebras depend continuously on a parameter $k$ such as the universal affine vertex algebra $V^k(\lie{g})$ associated to a simple, finite-dimensional Lie algebra $\lie{g}$. Another example is the $\mathcal{W}$-algebra $\mathcal{W}^k(\lie{g},f)$ associated to $\lie{g}$ together with a nilpotent element $f\in \lie{g}$. Typically, if $\mathcal{V}^k$ is such a vertex algebra depending on $k$, it is simple for generic values of $k$ but has a nontrivial maximal proper graded ideal $\mathcal{I}_k$ for special values. Often, one is interested in the structure and representation theory of the simple graded quotient $ \mathcal{V}^k / \cI_k$ at these points. This is illustrated by Frenkel and Zhu in \cite{FZ} to prove the $C_2$-cofiniteness and rationality of simple affine vertex algebras at positive integer level, and by Arakawa in \cite{A} to prove the $C_2$-cofiniteness and rationality of several families of $\mathcal{W}$-algebras.

Let $\cV^k$ be such a vertex algebra and let $G \subset \text{Aut}(\cV^k)$ be a reductive group of automorphisms. In addition to determining the generic structure of $(\cV^k)^G$, it is important to determine the \textit{nongeneric} set, where the strong generating set does not work. By a general result of \cite{CL}, this set is always finite and consists at most of the poles of the structure constants of the OPE algebra among the generators. Determining this set explicitly is not an easy problem even using a computer, although examples where it has been worked out appear in \cite{ACL,ACKL,AL}.

The primary objective of finding the nongeneric points is that it allows us to study orbifolds of the simple quotient $\cV_k$ of $\cV^k$, provided that $k$ is generic in the above sense. The quotient homomorphism $\cV^k \rightarrow \cV_k$ always restricts to a surjective homomorphism $$(\cV^k)^{G} \rightarrow (\cV_k)^{G},$$ so a strong generating set for $(\cV^k)^{G}$ descends to a strong generating set for $(\cV_k)^{G}$. In the examples we consider, the most interesting values of $k$, for which $(\cV)^G$ is highly reducible, turn out to be generic, so we obtain strong generators for $(\cV_k)^{G}$ as well.

In this paper, we study $V^k(\lie{g})^{\mathbb{Z}_{2}}$. Here $\lie{g}$ is a simple, finite-dimensional Lie algebra and $V^k(\lie{g})$ denotes the universal affine vertex algebra of $\mathfrak{g}$ at level $k$. There is an involution of $\lie{g}$ known as the {\it Cartan involution}, and it gives rise to the action of $\mathbb{Z}_2$ on $V^k(\lie{g})$. Let $l = \text{rank}(\lie{g})$ and let $m$ be the number of positive roots, so that $\text{dim}(\lie{g}) = 2m+ l$. Our main result is that for any $\lie{g}$ with $\text{dim}(\lie{g})> 3$, $V^{k}(\lie{g})^{\mathbb{Z}_{2}}$ is of type $$\mathcal{W}\big(1^{m},2^{d+ \binom{d}{2}},3^{ \binom{d}{2}},4\big),$$ for generic values of $k$. Here $d = m + l$. In this notation, we say that a vertex algebra is of type $\mathcal{W}((d_1)^{n_1},\dots (d_r)^{n_r})$ if it has a minimal strong generating set consisting of $n_i$ fields in weight $d_i$, for $i=1,\dots,r$. In the case $\lie{g} = \mathfrak{sl}_2$, there is one extra field in weight $4$, so that $V^k(\lie{g})^{\mathbb{Z}_{2}}$ is of type $\mathcal{W}(1,2^{3},3,4^{2})$ for generic values of $k$.

The proof of this result can be done by using a deformation argument \cite{LII,CL} in the sense that, $$\lim_{k\rightarrow \infty}  V^k(\lie{g})^{\mathbb{Z}_2} \cong \mathcal{H}(m) \otimes \big(\mathcal{H}(d)^{\mathbb{Z}_2}\big).$$ Here $\mathcal{H}(k)$ denotes the rank $k$ Heisenberg vertex algebra, and $\mathbb{Z}_2$ acts on the generators by multiplication by $-1$. Moreover, the limiting structure has a minimal strong generating set of the same type as $V^k(\lie{g})^{\mathbb{Z}_2}$ for generic values of $k$. So the problem of finding a minimal strong generating set for $V^k(\lie{g})^{\mathbb{Z}_2}$ is reduced to finding the minimal strong generating set for $\mathcal{H}(d)^{\mathbb{Z}_2}$ for all $d$. In the case $d = 1$, $\mathcal{H}(1)^{\mathbb{Z}_2}$ is of type $\mathcal{W}(2,4)$ by a celebrated theorem of Dong and Nagatomo \cite{DNI}. We will show that for $d = 2$, $\mathcal{H}(2)^{\mathbb{Z}_2}$ is of type $\mathcal{W}(2^3, 3, 4^2)$ and for $d \geq 3$, $\mathcal{H}(d)^{\mathbb{Z}_2}$ is of type $\mathcal{W}(2^{d+ \binom{d}{2}},3^{ \binom{d}{2}},4)$. To the best of our knowledge, the minimal strong generating set for $\mathcal{H}(d)^{\mathbb{Z}_2}$ does not appear previously in the literature. However, in \cite{DNII}, the representation theory $\mathcal{H}(d)^{\mathbb{Z}_2}$ was studied and the irreducible, positive-energy modules of $\mathcal{H}(d)^{\mathbb{Z}_2}$ were classified.
  
Finally, in the case $\lie{g}=\mathfrak{sl}_2$ we determine the set of nongeneric values; it consists only of $\{0,\frac{16}{51},\frac{16}{9},-\frac{32}{3}\}$. It follows that for all other values of $k$, the strong generating set for $V^k(\mathfrak{sl}_2)^{\mathbb{Z}_2}$ will descend to a strong generating set for the simple orbifold $L_k(\mathfrak{sl}_2)^{\mathbb{Z}_2}$. Here $L_k(\mathfrak{sl}_2)$ denotes the simple quotient of $V^k(\mathfrak{sl}_2)$.

\section{Preliminaries}

\textbf{Vertex algebras.} The notion of vertex algebra was introduced by Borcherds \cite{B} back in the eighties. Since then, it has been in a remarkable progress of study (see for example \cite{B,FLM,K,FBZ}). We will use the formalism developed in \cite{LZ} and \cite{LiI}. Roughly speaking, a \textit{vertex algebra} is a quantum operator algebra $\mathcal{A}$ in which any two elements $a, b$ are local. By \textit{local}, we mean that there exists some positive integer $N$ such that $(z - w)^{N} [a(z), b(w)] = 0.$ Here $\mathcal{A}$ is assumed to be a $\mathbb{Z}_{2}$-graded, and $[a(z),b(w)]$ is the super bracket, that is $[a(z),b(w)]=a(z)b(w)-(-1)^{|a||b|}b(w)a(z).$ There are many definitions for vertex algebra, and this is an equivalent definition in \cite{FLM}. Each $a \in \mathcal{A}$ has a unique representation by the following formal distribution:
$$a=a(z):=\sum_{n \in \mathbb{Z}} a(n) z^{-n-1} \in \textit{End} (V)[[z,z^{-1}]] .$$
The normally ordered product of fields in a vertex algebra $\mathcal{A}$ is called the Wick product, and is defined by
$$:a(z)b(w):\ =a(z)_{-} b(z) +(-1)^{|a||b|}b(z)a(z)_{+},$$ 
where $a(z),b(w) \in \mathcal{A},$ and $$a(z)_{-} =\sum _{n<0}a(n)z^{-n-1} \enspace \enspace \enspace a(z)_{+} = \sum _{n \geq  0} a(n) z^{-n-1}.$$
The $k$-fold iterated Wick product is defined inductively as follows:
$$:a_{1}(z)\cdots a_{n}(z):\ =\ :a_{1}(z)(:a_{2}(z)\cdots a_{n}(z):):,$$
for $a_{1}(z),\dots,a_{n}(z) \in \mathcal{A}.$ 
The operators product expansion (OPE) formula for $a,b \in \mathcal{A}$ is defined by
\begin{equation}
a(z)b(w) \sim \sum _{n \geq  0}a(w) \circ _{n} b(w)( z-w)^{-n-1},
\end{equation}
where $\sim$ means equal modulo terms that are regular at $z=w.$ Here $\circ_{n}$ denotes the $n^{th}$ circle product, which is defined by
$$a(w) \circ _{n} b(w)= Res_{z}a(z)b(w)\imath _{|z|>|w|}(z-w)^{n}-(-1)^{|a||b|}Res_{z}b(w)a(z)\imath _{|w|>|z|}(z-w)^{n},$$
where $\imath _{|z|>|w|} f(z,w) \in \mathbb{C}[[z,z^{-1},w,w^{-1}]]$ denotes the power series expansion of a rational function $f$ which converges in the domain $|z|>|w|,$ while $Res_{z}a(z)$ denotes the coefficient of $z^{-1}.$ The $\partial_{z} a(z)$ is the formal derivative $\partial_{z}=\frac{d}{dz}.$\newline 

 A subset $S=\{a_{i}|i \in I\} \subset \mathcal{A}$ is said to \textit{generate} $\mathcal{A}$ if every $a \in \mathcal{A}$ is a linear combination of words in $a_{i},\circ_{n}$ for $i \in I$ and $n\in \mathbb{Z}.$ Moreover, we say that $S$  \textit {strongly generates} $\mathcal{A}$ if every $a \in \mathcal{A}$ is a linear combination of words in $a_{i},\circ_{n}$ for $n<0.$ Equivalently, $\mathcal{A}$ is spanned by the ordered monomials 
 \begin{equation}\label{monomial}
\{:\partial^{k_{1}} a_{i_{1}} \cdots \partial^{k_{m}} a_{i_{m}}:\ |i_{1},\dots ,i_{m}\in I, \enspace 0\leq k_{1}\leq \cdots \leq k_{m}\}.
\end{equation}
If $I$ can be chosen to be finite, then $\mathcal{A}$ is called \textit {strongly finitely generated}. \newline
\hphantom{X} We say that $S$ \textit{ freely generates} $\mathcal{ A}$ if there are no nontrivial normally ordered polynomial relations among the generators and their derivatives. \newline

Our work hinges mainly on the \textit{universal affine vertex algebras}. Let $\mathfrak{g}$ be a finite-dimensional Lie algebra over $\mathbb{C},$ furnished with a nondegenerate, symmetric, invariant bilinear form $B$. The \textit{affine Kac-Moody algebra} $\hat{\lie{g}}=\lie{g}[t,t^{-1}] \oplus \mathbb{C}_{\kappa},$ determined by $B,$ is the one-dimensional central extension of the loop algebra $\lie{g}[t,t^{-1}]=\lie{g}\otimes \mathbb{C}[t,t^{-1}],$ where a generator ${\kappa}$ is the central charge. The Lie algebra $\hat{\lie{g}}$ is spanned by $\langle\zeta  \otimes t^{n},\zeta  \in \lie{g},n \in \mathbb{Z},\kappa \rangle,$ and these generators satisfy the following Lie bracket:
\begin{equation}\label{general lie algebra}
[\zeta \otimes t^{n},\eta  \otimes t^{m}]=[\zeta,\eta]  \otimes t^{n+m}+nB(\zeta, \eta)\delta_{n+m,0}\kappa,\enspace\enspace\enspace [\kappa,\zeta  \otimes t^{n}]=0,
\end{equation}
and $\mathbb{Z}$-gradation $deg(\zeta \otimes t^{n})=n,$ and $deg(\kappa)=0.$
\newline
\hphantom{X} Let $\hat{\lie{g}}_{+}=\sum_{n\geq 0}\hat{\lie{g}}_{n}$ where $\hat{\lie{g}}_{n}$ denotes the subspace of degree n. For $k \in \mathbb{C},$ let $C_{k }$ be the one-dimensional $\hat{\lie{g}}_{+}$-module on which $\zeta \otimes t^n$ acts trivially for $n \geq 0,$ and $\kappa$ as a multiplication by scalar $k.$ Define $V_{k} = U(\hat{\lie{g}}) \otimes _{U(\hat{\lie{g}}_{+})} C_{k},$ and let $X^{\zeta}(n) \in End(V_{k})$ be the linear operator representing $\zeta \otimes t^{n}$ on $V_{k}.$ Define $X^{\zeta}(z)=\sum_{n \in \mathbb{Z}}X^{\zeta}(n)z^{-n-1}$ to be an even generating field of conformal weight $1,$ and satisfies the OPE relation
 \begin{equation}\label{ope affine universal}
 X^{\zeta}(z)X^{\eta}(w)\sim kB(\zeta,\eta)(z-w)^{-2}+X^{[\zeta,\eta]}(w)(z-w)^{-1}.
\end{equation}

 The vertex algebra $V^{k}(\lie{g},B)$ generated by $\{X^{\zeta_{i}}|\zeta_{i} \in \lie{g}\}$ is called the universal affine vertex algebra associated to $\mathfrak{g}$ and $B$ at level $k$. It has a PBW basis
\begin{equation}\label{basis g universal}
:\partial ^{k_{1}^{1}} X^{\zeta_{1}}\dots \partial ^{k_{s_{1}}^{1}} X^{\zeta_{1}} \dots \partial ^{k_{1}^{m}} X^{\zeta_{m}} \dots \partial^{k_{s_{m}}^{m}} X^{\zeta_{m}} :, \enspace \enspace s_{i}\geq0,  \enspace k_{1}^{i} \geq \dots \geq k_{s_{i}}^{i} \geq 0. 
\end{equation}

A special case is when $\lie{g}$ is a simple Lie algebra, the bilinear form $B$ is then defined to be the normalized Killing form.
Then \eqref{general lie algebra} can be defined in this case as follows:
$$[\zeta \otimes t^{n},\eta  \otimes t^{m}]=[\zeta,\eta]  \otimes t^{n+m}+n(\zeta|\eta)\delta_{n+m,0}\kappa,\enspace\enspace\enspace [\kappa,\zeta  \otimes t^{n}]=0,$$
for $\zeta, \eta \in \lie{g},n,m \in \mathbb{Z}.$ Here $(.|.)$ is \textit{the normalized Killing form}, and is defined as 
\begin{equation*}
(.|.)=\frac{1}{2h^{\vee}}(.,.)_{\kappa_{\lie{g}}},
\end{equation*}
where $h^{\vee}$ is the dual Coxeter number of $\lie{g}.$ In this case, we denote $V^{k}(\mathfrak{g},B)$ by $V^{k}(\mathfrak{g}).$

According to the Cartan-Killing classification, we have the following list for the simple, finite-dimensinal Lie algebras over $\mathbb{C}$.

\noindent \textbf{Classical Lie Algebras:}
\begin{enumerate}
\item $\mathfrak{sl}_{n+1},n\geq 1,$ and it has Cartan notation $A_{n},$
\item $\mathfrak{so}_{2n+1},n\geq 2,$ and it has Cartan notation $B_{n},$
\item $\mathfrak{sp}_{2n},n\geq 3,$ and it has Cartan notation $C_{n},$
\item $\mathfrak{so}_{2n},n\geq 4,$ and it has Cartan notation $D_{n}.$
\end{enumerate}
\textbf{Exceptional Lie Algebras:}
\begin{enumerate}
\item $G_{2},$
\item $F_{4},$
\item $E_{6},E_{7},$ or $E_{8}.$
\end{enumerate}

Let $\{\zeta_{1},\dots ,\zeta_{n}\}$ be an orthonormal basis for $\lie{g}$ relative to $(.|.).$ There is a natural conformal structure of central charge $\frac{k \cdot \text{dim}(\lie{g})}{k+h^{\vee}}$ on $V^{k}(\mathfrak{g})$ with the Virasoro element $L(z),$ that is
\begin{equation}
L(z)=\frac{1}{2(k+h^{\vee})}\sum_{i=1}^{n}:X^{\zeta_{i}}(z)X^{\zeta_{i}}(z):,
\end{equation}
where $k\neq -h^{\vee}.$ In this case, the Virasoro element is called the \textit{Sugawara conformal vector.} For $k= -h^{\vee},$  the Virasoro element $L(z)$ does not exist.

For the case where $\lie{g}$ is an abelian Lie algebra. Since $B$ is nondegenerate, $V^{k}(\mathfrak{g},B)$ is just the rank $n$ Heisenberg vertex algebra $\mathcal{H}(n).$  If we choose an orthonormal basis $\{\zeta_{1},\dots,\zeta_{n}\}$ for $\lie{g},$ then $\mathcal{H}(n)$ is generated by $\{\alpha^{i}= X^{\zeta_{i}}|i=1,\dots,n\}.$ 
\newline

\textbf{ Good increasing filtrations.} \cite{LiII} A good increasing filtration on a vertex algebra $\mathcal{A}$ is a $\mathbb{Z}_{\geq0}$-filtration
\begin{equation}\label{eq11}
\mathcal{A}_{(0)} \subset \mathcal{A}_{(1)}\subset \mathcal{A}_{(2)}\cdots ,\enspace \enspace\enspace \mathcal{A}=\bigcup _{d\geq 0}\mathcal{A}_{(d)}
\end{equation}
satisfying that $\mathcal{A}_{(0)}= \mathbb{C}$, and for all $a \in \mathcal{ A}_{(r)}, b  \in \mathcal{ A}_{(s)}$ we have
\begin{equation}
a \circ_{n} b \in \mathcal{ A}_{(r+s)}, \enspace \enspace \text{for} \enspace n < 0,
\end{equation}
\begin{equation}
a \circ_{n} b \in \mathcal{ A}_{(r+s-1)}, \enspace \enspace  \text{for} \enspace n \geq 0.
\end{equation}
Let $\mathcal{A}_{-1}=\{0\}.$ An element $a(z)\in \mathcal{A}$ has at most degree $d$ if $a(z) \in  \mathcal{A}_{(d)}.$

The associated graded algebra $\text{gr}(\mathcal{A}) =\bigoplus_{d\geq0} \mathcal{A}_{(d)}/\mathcal{A}_{(d-1)}$ is a $\mathbb{Z}_{\geq0}$-graded associative, (super)commutative algebra with a unit 1 under a product induced by the Wick product on $\mathcal{A}.$ It has a derivation $\partial$ of degree zero.\newline
 For each $r \geq 1,$ we have the projections
 \begin{equation}
 \phi_{r} : \mathcal{A}_{(r)} \rightarrow \mathcal{A}_{(r)}/\mathcal{A}_{(r-1)} \subset \text{gr}(\mathcal{A}).
 \end{equation}
\hphantom{X} Let $\mathcal{ R}$ be the category of vertex algebras associated with a $\mathbb{Z}_{\geq0}$-filtration. For any vertex algebra $\mathcal{A} \in \mathcal{ R},$ the $\partial$-ring, namely $\text{gr}(\mathcal{A})$ 		is an abelian vertex algebra. The following reconstruction property is the key peculiarity of $\mathcal{R}$ \cite{LL}:
\begin{lemma}\label{8} Given a vertex algebra $\mathcal{A} \in \mathcal{R}.$ Consider the collection $\{a_{i}| i \in I\}$ that generates $\text{gr}(\mathcal{A})$ as a $\partial$-ring, where $a_{i}$ is homogenous of degree $d_{i}.$ Then $\mathcal{A}$ is strongly generated by the collection $\{a_{i}(z)| i \in I\},$ where $a_{i}(z) \in \mathcal{A}_{(d_{i})}$ such that $\phi_{d_{i}}(a_{i}(z))=a_{i}.$
\end{lemma}

We define an increasing filtration on $V^{k}(\lie{g})$ for any simple Lie algebra $\lie{g}$ as follows:\newline
\begin{equation}\label{filg}
V^{k}(\lie{g})_{(0)} \subset V^{k}(\lie{g})_{(1)} \subset  \cdots, \enspace \enspace \enspace V^{k}(\lie{g})=\bigcup_{r\geq 0} V^{k}(\lie{g})_{(r)} ,
\end{equation}
where $V^{k}(\lie{g})_{(-1)}=\{0\},$ and $V^{k}(\lie{g})_{(r)}$ is spanned by the iterated Wick products of the generators $X^{\zeta_{i}}$ and their derivatives, such that at most $r$ of the generators and their derivatives appear.  \newline

So, $V^{k}(\lie{g})$ equipped with this filtration lies in the category $\mathcal{R}.$ The $\mathbb{Z}_{\geq 0}$-associated graded algebra $$\text{gr}(V^{k}(\lie{g}))=\bigoplus_{d\geq 0}V^{k}(\lie{g})_{(d)}/V^{k}(\lie{g})_{(d-1)}$$ is now an abelian vertex algebra freely generated by $X^{\zeta_{i}}.$ Then, $V^{k}(\lie{g})\cong \text{gr}( V^{k}(\lie{g}))$ as linear spaces, and as commutative algebras we have 
\begin{equation*}
\text{gr}(V^{k}(\lie{g}))\cong \mathbb{C}[ X^{\zeta_{i}},\partial X^{\zeta_{i}},\partial^{2}X^{\zeta_{i}},\dots ].
\end{equation*}

\section{The $\mathbb{Z}_{2}$-orbifold of $\mathcal{H}(n)$}

The rank $n$ Heisenberg vertex algebra $\mathcal{H}(n)$ is the tensor product of $n$ copies of rank $1$ Heisenberg vertex algebra $\mathcal{H}$ with even generating fields $\alpha^{1},\dots,\alpha^{n}.$ These satisfy the OPE relations
\begin{equation}\label{nheisenberg}
\alpha^{i}(z)\alpha^{j}(w)\sim \delta _{i,j} (z-w)^{-2}.
\end{equation}
 There is a conformal structure of central charge $n$ on $\mathcal{H}(n)$ with the Virasoro element $L(z)$
\begin{equation}
L(z)=\frac{1}{2}\sum_{i=1}^{n}:\alpha^{i}(z)\alpha^{i}(z):,
\end{equation}
where each $\alpha^{i}$ is primary of weight $1$.

$\mathcal{H}(n)$ is freely generated by $\alpha^{1},\dots,\alpha^{n}$ and has a PBW basis as follows
\begin{equation}\label{eq22}
:\partial ^{k_{1}^{1}} \alpha^{1}\dots \partial ^{k_{s_{1}}^{1}} \alpha^{1} \dots \partial ^{k_{1}^{n}} \alpha^{n} \dots \partial^{k_{s_{n}}^{n}} \alpha ^{n} :, \enspace \enspace s_{i}\geq0,  \enspace k_{1}^{i} \geq \dots \geq k_{s_{i}}^{i} \geq 0. 
\end{equation}

 \textbf{Filtrations on $\mathcal{H}(n)$.} Define an increasing filtration on $\mathcal{H}(n)$ as follows: 
\begin{equation}\label{filh}
\mathcal{H}(n)_{(0)} \subset \mathcal{H}(n)_{(1)} \subset \mathcal{H}(n)_{(2)} \subset \cdots ,\enspace \enspace \enspace \mathcal{H}(n)=\bigcup _{d\geq 0} \mathcal{H}(n)_{(d)} ,
\end{equation}
where $\mathcal{H}(n)_{(-1)}=\{0\},$ and $\mathcal{H}(n)_{(r)}$ is spanned by the iterated Wick products of the generators $\alpha^{i}$ and their derivatives such that at most $r$ of $\alpha^{i}$ and their derivatives appear.

From defining the OPE relation \eqref{nheisenberg}, this is a good increasing filtration, and so, $\mathcal{H}(n)$ equipped with such a good filtration lies in the category $\mathcal{R}.$ The OPE relation will be replaced with $\alpha^{i}(z)\alpha^{j}(w)\sim 0.$ The $\mathbb{Z}_{\geq 0}$-associated graded algebra $$\text{gr}(\mathcal{H}(n))=\bigoplus_{d\geq0} \mathcal{H}(n)_{(d)}/\mathcal{H}(n)_{(d-1)}$$ is now an abelian vertex algebra freely generated by $\alpha^{i}.$ The rank $n$ Heisenberg vertex algebra $\mathcal{H}(n)$ equipped with such filtrations lies in the category $\mathcal{R}.$ Then, $\mathcal{H}(n)\cong \text{gr}( \mathcal{H}(n))$ as linear spaces, and as commutative algebras, we have 
\begin{equation*}
\text{gr}(\mathcal{H}(n))\cong \mathbb{C}[ \partial^{a} \alpha^{i}|a\geq0,i=1,\dots,n].
\end{equation*}

The subgroup $\mathbb{Z}_{2}$ of the automorphism group of $\mathcal{H}(n)$ generated by the nontrivial involution $\theta$ which acts on the generators as follows:
\begin{equation}\label{action of Z2H(n)}
\theta(\alpha^{i})=  -\alpha^{i}.
\end{equation}
The OPE relations \eqref{nheisenberg} will be preserved by this action on $\mathcal{H}(n),$ that is  
 $$\alpha^{i} \circ_{m}\alpha^{j}=\theta(\alpha^{i})\circ_{m} \theta( \alpha^{j})$$ for all $m$ as well as the filtration \eqref{filh}, and induces an action of $\mathbb{Z}_{2}$ on $\text{gr}(\mathcal{H}(n)).$

Going back to $\mathcal{H}(n)^{\mathbb{Z}_{2}},$ it is also spanned by all normally ordered monomials of the form \eqref{eq22}, where the length $s_1  + \cdots + s_n$ is even. Since $\mathcal{H}(n)$ is freely generated by $\alpha^{i}$, these monomials form a basis for $\mathcal{H}(n)^{\mathbb{Z}_{2}}$, and the normal form is unique.

The filtration on $\mathcal{H}(n)^{\mathbb{Z}_{2}}$ is obtained from the filtration \eqref{filh} after restriction as follows: 
$$(\mathcal{H}(n)^{\mathbb{Z}_{2}})_{(0)} \subset (\mathcal{H}(n)^{\mathbb{Z}_{2}})_{(1)}\subset \cdots ,\enspace \enspace \enspace \mathcal{H}(n)^{\mathbb{Z}_{2}}=\bigcup _{d\geq 0} (\mathcal{H}(n)^{\mathbb{Z}_{2}})_{(d)},$$
 where $(\mathcal{H}(n)^{\mathbb{Z}_{2}})_{(r)}=\mathcal{H}(n)^{\mathbb{Z}_{2}} \cap \mathcal{H}(n)_{(r)}.$

 The action of $\mathbb{Z}_{2}$ on $\mathcal{H}(n)$ descends to an action on $\text{gr}(\mathcal{H}(n))$, and so we have a linear isomorphism $ \mathcal{H}(n)^{\mathbb{Z}_{2}}\cong \text{gr}(\mathcal{H}(n)^{\mathbb{Z}_{2} })$ as linear spaces. Similarly, $\mathbb{Z}_{2}$ acts on $\text{gr}(\mathcal{H}(n))\cong \mathbb{C}[\partial^{a} \alpha^{i}|a\geq 0,i=1,\dots ,n],$ and so we have a linear isomorphism 
 \begin{equation}\label{grZ2}
\text{gr}(\mathcal{H}(n)^{\mathbb{Z}_{2}}) \cong \text{gr}(\mathcal{H}(n))^{\mathbb{Z}_{2}} \cong \mathbb{C}[\partial^{a} \alpha^{i}|a\geq0,i=1,\dots ,n]^{\mathbb{Z}_{2}}
\end{equation}
 as commutative algebras. The weight and degree are preserved by \eqref{grZ2} where $wt(\partial^{a}\alpha^{i})=a+1.$

Recall, the $\text{gr}(\mathcal{H}(n))^{\mathbb{Z}_{2}}$ is a commutative algebra of even degree with a differential $\partial$ of degree zero, and it extends to $\text{gr}(\mathcal{H}(n))^{\mathbb{Z}_{2}}$ by the product rule, that is
\begin{align*}
\partial (\alpha_{a}^{i} \alpha_{b}^{j})  =\alpha_{a+1}^{i} \alpha_{b}^{j}+\alpha_{a}^{i} \alpha_{b+1}^{j}.
\end{align*}

Define 
 \begin{equation*}
 q^{i,i}_{a,b}=\alpha_{a}^{i} \alpha_{b}^{i},\enspace\enspace\enspace q^{i,j}_{a,b}=\alpha_{a}^{i} \alpha_{b}^{j},
 \end{equation*}
as generators for $\text{gr}(\mathcal{H}(n))^{\mathbb{Z}_{2}}.$ The action of $\partial $ on these generators is defined as follows:
\begin{equation}\label{action of derivation Hn}
 \partial(q^{i,i}_{a,b}) =q^{i,i}_{a+1,b}+ q^{i,i}_{a,b+1}, \enspace \enspace \enspace \partial (q_{a,b}^{i,j}) =q^{i,j}_{a+1,b}+ q^{i,j}_{a,b+1}.
\end{equation}
 The action of $\mathbb{Z}_{2}$ on the $\text{gr}(\mathcal{H}(n))$ which is given by $\theta(\alpha_{a}^{i} )=-\alpha_{a}^{i} $ guarantees that $\text{gr}(\mathcal{H}(n))^{\mathbb{Z}_{2}}$ is generated by the subset $\{q^{i,i}_{a,b},q^{i,j}_{a,b}| a,b\geq0,\enspace 1 \leq i, j\leq n \}.$ Since $q^{i,i}_{a,b}=q^{i,i}_{b,a},$ and $q^{i,j}_{a,b}=q^{j,i}_{b,a},$ so $\text{gr}(\mathcal{H}(n))^{\mathbb{Z}_{2}}$ is generated by the subset 
 \begin{equation}\label{generating set}
 \{q^{i,i}_{a,b}| 0 \leq a \leq b,\enspace i=1,\dots ,n\} \bigcup \{q^{i,j}_{a,b}|  0 \leq a , b,\enspace 1 \leq i < j \leq n\}.
 \end{equation}

 Among these generators, the ideal of relations is generated by 
\begin{equation} \label{ideal'} q^{i,j}_{r,s} q^{k,l}_{t,u} - q_{r,u}^{i,l} q_{s,t}^{j,k},\enspace  i,j,k,l=1,\dots,n, \enspace 0 \leq r,s,t,u.
\end{equation}
Under the projection 
\begin{equation*}
\phi_{2}:(\mathcal{H}(n)^{\mathbb{Z}_{2}})_{(2)}\rightarrow (\mathcal{H}(n)^{\mathbb{Z}_{2}})_{(2)}/(\mathcal{H}(n)^{\mathbb{Z}_{2}})_{(1)}\subset \text{gr}(\mathcal{H}(n)^{\mathbb{Z}_{2}}),
\end{equation*}
the generators $q^{i,i}_{a,b},q_{a,b}^{i,j}$ of $\text{gr}(\mathcal{H}(n))^{\mathbb{Z}_{2}}$ correspond to fields $ \omega ^{i,i}_{a,b}, \omega _{a,b}^{i,j},$ respectively defined by
\begin{equation}
\omega_{a,b}^{i,i}= \ :\partial^{a} \alpha^{i}(z) \partial^{b} \alpha^{i}(z):\ \in (\mathcal{H}(n)^{\mathbb{Z}_{2}})_{(2)},\enspace \enspace \enspace  0\leq a \leq b ,\enspace \enspace i = 1,\dots, n, 
\end{equation}
\begin{equation}
\omega^{i,j}_{a,b}= \ :\partial^{a} \alpha^{i}(z) \partial^{b} \alpha^{j}(z): \ \in (\mathcal{H}(n)^{\mathbb{Z}_{2}})_{(2)},\enspace \enspace \enspace a,b \geq 0, \enspace\enspace  1\leq i < j \leq n.
\end{equation}
The fields $ \omega _{a,b}^{i,i}, \omega _{a,b}^{i,j}$ satisfy $\phi_{2}(\omega _{a,b}^{i,i})=q^{i,i}_{a,b},$ $\phi_{2}(\omega _{a,b}^{i,j})=q_{a,b}^{i,j},$ respectively and have weight $a+b+2.$ Note that $ \sum_{i=1}^{n}\omega_{0,0}^{i,i}=2L,$ where $L$ is the Virasoro element. The subspace $(\mathcal{H}(n)^{\mathbb{Z}_{2}})_{(2)}$ has degree at most $2,$ and has a basis $\{1\} \cup \{ \omega^{i,i} _{a,b},\omega _{a,b}^{i,j}\}.$ Moreover, for $m \geq 0,$ the operators $\omega _{a,b}^{i,j} \circ_{m}$ preserve this vector space \cite{LI}. For $a,b,c \geq0,\enspace 0 \leq m \leq a+b+c+1,$ and $i < j,$ we have
\begin{equation} 
\omega _{a,b}^{i,j} \circ _{m} \partial^{c} \alpha ^{i}=(-1)^{a} \frac{(a+c+1)!}{(a+c+1-m)!} \partial ^{a+b+c+1-m} \alpha ^{j},
\end{equation} 
\begin{equation} 
\omega _{a,b}^{i,j} \circ _{m} \partial^{c} \alpha ^{j}=(-1)^{b} \frac{(b+c+1)!}{(b+c+1-m)!} \partial ^{a+b+c+1-m} \alpha ^{i}.
\end{equation} 
\begin{equation}
\omega _{a,b}^{i,i} \circ _{m} \partial^{c} \alpha ^{i}=\lambda _{a,b,c,m} \partial ^{a+b+c+1-m} \alpha ^{i},
\end{equation}
where
\begin{equation*}
\lambda _{a,b,c,m}=(-1)^{b} \frac{(b+c+1)!}{(b+c+1-m)!}+(-1)^{a}\frac{(a+c+1)!}{(a+c+1-m)!}.
\end{equation*}
It follows that for $m \leq a+b+c+1,$ and $i < j < k$ we have
\begin{align}
\omega _{a,b}^{i,j} \circ _{m} \omega _{c,d}^{i,j} &= (-1)^{a} \frac{(a+c+1)!}{(a+c+1-m)!} \omega _{a+b+c+1-m,d}^{j,j}  \notag \\
&+(-1)^{b} \frac{(b+d+1)!}{(b+d+1-m)!}\omega _{a+b+d+1-m,c}^{i,i} ,
\end{align}
\begin{equation}\label{wijjk}
\omega _{a,b}^{i,j} \circ _{m} \omega _{c,d}^{j,k} = (-1)^{b} \frac{(b+c+1)!}{(b+c+1-m)!}  \omega _{a+b+c+1-m,d}^{i,k} ,
\end{equation}
\begin{equation}\label{wiiij}
\omega _{a,b}^{i,i} \circ _{m} \omega _{c,d}^{i,j} = \lambda_{a,b,c,m} \omega _{a+b+c+1-m,d}^{i,j} ,
\end{equation}
\begin{equation}\label{wjjij}
\omega _{a,b}^{j,j} \circ _{m} \omega _{c,d}^{i,j} = \lambda_{a,b,d,m} \omega _{c,a+b+d+1-m}^{i,j},
\end{equation}
\begin{equation}\label{wiiii}
\omega _{a,b}^{i,i} \circ _{m} \omega ^{i,i}_{c,d}= \lambda _{a,b,c,m} \omega^{i,i} _{a+b+c+1-m,d}  +\lambda _{a,b,d,m}\omega^{i,i} _{c,a+b+d+1-m}.
\end{equation}

As a differential algebra with derivation $\partial$, some of the generators in the generating
set \eqref{generating set} for $\text{gr}(\mathcal{H}(n))^{\mathbb{Z}_{2}}$ can be eliminated due to \eqref{action of derivation Hn}. For $m\geq0,$ let $$A_{m}=\text{span}\{\omega _{a,b}^{i,i}|a+b=m\}$$ be the vector space which is homogenous of weight $m+2$. Using the relation $\partial \omega _{a,b}^{i,i}=\omega_{a+1,b}^{i,i}+\omega _{a,b+1}^{i,i}$, we see that $\text{dim}(A_{2m})=m+1=\text{dim}(A_{2m+1}),$ for $m \geq 0.$ Moreover, $\partial(A_{m}) \subset A_{m+1},$ and 
\begin{equation}
\text{dim}(A_{2m}/\partial(A_{2m-1}))=1, \enspace \enspace \enspace  \text{dim}(A_{2m+1}/\partial(A_{2m}))=0.
\end{equation}
Thus, $A_{2m}$ has a decomposition
\begin{equation}
A_{2m}=\partial(A_{2m-1}) \oplus \langle \omega_{0,2m}^{i,i} \rangle = \partial^{2}(A_{2m-2}) \oplus \langle \omega_{0,2m}^{i,i} \rangle,
\end{equation}
where $\langle \omega_{0,2m}^{i,i}\rangle$ is the linear span of $ \omega^{i,i}_{0,2m}.$ Similarly, $A_{2m+1}$ has a decomposition
\begin{equation}
A_{2m+1}=\partial^{2}(A_{2m-1}) \oplus \langle \partial \omega_{0,2m}^{i,i} \rangle = \partial^{3}(A_{2m-2}) \oplus \langle \partial \omega_{0,2m}^{i,i} \rangle.
\end{equation}
Therefore, 
\begin{equation*} 
\text{span} \{\omega^{i,i}_{a,b}|a+b=2m\}=\text{span}\{\partial ^{2k} \omega^{i,i}_{0,2m-2k}|0\leq k\leq m\}
\end{equation*} 
and 
\begin{equation*}
\text{span} \{\omega^{i,i}_{a,b}|a+b=2m+1\}=\text{span}\{\partial ^{2k+1} \omega^{i,i}_{0,2m-2k}|0\leq k\leq m\}
\end{equation*}
 are bases of $A_{2m}$ and $A_{2m+1},$ respectively and so for each $\omega_{a,b}^{i,i}\in A_{2m}$ and $\omega_{c,d}^{i,i} \in A_{2m+1}$ can be written uniquely in the form 
\begin{equation}\label{linear}
\omega_{a,b}^{i,i}=\sum_{k= 0}^{m}  \lambda_{k} \partial^{2k} \omega^{i,i}_{0,2m-2k}, \enspace \enspace 
\omega_{c,d}^{i,i}=\sum_{k= 0}^{m} \mu_{k} \partial^{2k+1} \omega^{i,i}_{0,2m-2k},
\end{equation}
for constants $\lambda_{k},\mu_{k}.$

Similarly, for $m\geq 0$, let $A'_{m}=\text{span}\{\omega _{a,b}^{i,j}|a+b=m\},$ and use the relation $\partial \omega _{a,b}^{i,j}=\omega_{a+1,b}^{i,j}+\omega _{a,b+1}^{i,j}.$ We have $\text{dim}(A'_{m})=m+1,$ for $m \geq0.$ Moreover, $\partial(A'_{m}) \subset A'_{m+1},$ and 
\begin{equation}
\text{dim}(A'_{m}/\partial(A'_{m-1}))=1.
\end{equation}
Hence, $A'_{m}$ has a decomposition
\begin{equation}
A'_{m}=\partial(A'_{m-1}) \oplus \langle \omega^{i,j}_{0,m}\rangle 
\end{equation}
where $\langle \omega^{i,j}_{0,m}\rangle$ is the linear span of $ \omega^{i,j}_{0,m}.$ Therefore, $$\text{span} \{\omega^{i,i}_{a,b}|a+b=m\}=\text{span}\{\partial ^{k} \omega^{i,j}_{0,m-k}|\\0\leq k\leq m\}$$ is a basis of $A'_{m}.$ It follows that for each $\omega_{r,m-r}^{i,j} \in A'_{m}$ can be written uniquely in the form 
\begin{equation}\label{linear combination of w}
\omega_{r,m-r}^{i,j}=\sum_{k= 0}^{r}  (-1)^{r+k}\binom {r}{k}\partial^{k} \omega^{i,j}_{0,m-k}, 
\end{equation}
where $r=0,\dots, m.$

The following lemma gives a strong generating set for $\mathcal{H}(n)^{\mathbb{Z}_{2}}$, which as we shall see is far from minimal.
\begin{lemma}\label{strong generators for H(n)}
$\mathcal{H}(n)^{\mathbb{Z}_{2}}$ is strongly generated as a vertex algebra by the subset
\begin{equation}\label{geneh}
 \{\omega_{0,2m}^{i,i}|m\geq0, \enspace and \enspace i=1,\dots, n\}\bigcup \{\omega^{i,j}_{0,m}|m\geq0, \enspace  and \enspace 1\leq i<j \leq n \}.
 \end{equation}

 \end{lemma}
\begin{proof} 
Since $\text{gr}(\mathcal{H}(n))^{\mathbb{Z}_{2}}=\text{gr}(\mathcal{H}(n)^{\mathbb{Z}_{2}})$ is generated by the subset $$\{q_{0,2m}^{i,i}|m\geq 0 ,\enspace i=1,\dots ,n\}  \\
\bigcup \{q_{0,m}^{i,j}|m\geq0, \enspace and \enspace 1\leq i<j \leq n\}$$ as a $\partial$-ring, Lemma \ref{8} shows that the corresponding set strongly generates $\mathcal{H}(n)^{\mathbb{Z}_{2}}$ as a vertex algebra.
\end{proof}

\section{Minimal strong generating set for $\mathcal{H}(n)^{\mathbb{Z}_2}$}\label{decoupling hEIENBERG relations}
In this section, we give a minimal strong generating set for $\mathcal{H}(n)^{\mathbb{Z}_2}$. First, we recall the case $n=1$, which is due to Dong and Nagatomo \cite{DNI}. For simplicity of notation, we write $\omega_{a,b} = \omega^{1,1}_{a,b}$ and $q_{a,b} = q^{1,1}_{a,b}$ in this case, and we include the proof for the benefit of the reader.

\begin{thm} (Dong-Nagatomo) $\mathcal{H}(1)^{\mathbb{Z}_2}$ has a minimal strong generating set $ \{\omega_{0,0},\omega_{0,2}\}$ and is of type $\mathcal{W}(2,4).$
\end{thm}

\begin{proof} 
Among the generators $\{q_{0,2m}|\ m\geq0\}$ of $\text{gr}(\mathcal{H}(1))^{\mathbb{Z}_{2}}$, the first relation of the form \eqref{ideal'} occurs of minimal weight $6$, and has the form
\begin{equation}\label{deal'1}
q_{0,0} q _{1,1} -q_{0,1} q_{0,1}=0.
\end{equation}
This relation is unique up to scalar. The corresponding element $: \omega_{0,0} \omega_{1,1} :-: \omega_{0,1}\omega _{0,1} :$ lies in $(\mathcal{H}(n)^{\mathbb{Z}_{2}})_{(2)}$. This element does not vanish, but it has a correction of the form
\begin{equation}\label {H1 rank}
: \omega_{0,0} \omega _{1,1} :-: \omega_{0,1}\omega _{0,1}: \ =-\frac{5}{4}  \omega _{0,4}+\frac{7}{4} \partial ^{2} \omega_{0,2}-\frac{7}{24} \partial ^{4} \omega_{0,0} .
 \end{equation}
Furthermore, we have
 \begin{equation}\label{w0111}
 \omega_{0,1}=\frac{1}{2} \partial  \omega_{0,0},\qquad \omega_{1,1}= - \omega_{0,2}+\frac {1}{2} \partial^{2}  \omega_{0,0}.
 \end{equation}
 Thus, \eqref{H1 rank} can be rewritten in the form
 \begin{equation}\label{PH14}
 \omega_{0,4}=P_4(\omega_{0,0},\omega_{0,2}),
 \end{equation}
where $P_4(\omega_{0,0},\omega_{0,2})$ is a normally ordered polynomial in $\omega_{0,0}, \omega_{0,2}$, and their derivatives. This is called a \textit{decoupling relation}, as $\omega _{0,4}$ can then be expressed as a normally ordered polynomial in $\omega_{0,0},\omega_{0,2}$ and their derivatives. 

Next, we can construct decoupling relations
\begin{equation} \label{n=1:higherdecoup} \omega_{0,2m} = P_{2m}(\omega_{0,0}, \omega_{0,2}),\end{equation} expressing $\omega_{0,2m}$ as a normally ordered polynomial in $\omega_{0,0}, \omega_{0,2}$ and their derivatives, for all $m>2$. We need the calculation
\begin{equation}
\omega_{0,2} \circ_{1} \omega_{0,2k}=(8+4k) \omega_{0,2k+2} +\partial^{2} \mu,
\end{equation}
where $\mu$ is a linear combination of $\partial ^{2r}\omega_{0,2k-2r}$ for $r=0,\dots ,k$. We can then construct the relations \eqref{n=1:higherdecoup} inductively by applying the operator $\omega_{0,2}  \circ_1$ repeatedly to \eqref{PH14}.

It follows that $\mathcal{H}(1)^{\mathbb{Z}_2}$ is strongly generated by $\{\omega_{0,0}, \omega_{0,2}\}$. To see that this is a {\it minimal} strong generating set, it suffices to observe that no decoupling relations for $\omega_{0,0},\omega_{0,2}$ can be found since there are no relations of weight less than $6$ in $\text{gr}(\mathcal{H}(1))^{\mathbb{Z}_{2}}$ of the form \eqref{ideal'}.
\end{proof}

The main result in this section is
\begin{thm} \label{strong generators for Hn}
\begin{enumerate}
\item For $n=2,$ $\mathcal{H}{(2)}^{\mathbb{Z}_{2}}$ has a minimal strong generating set
\begin{equation}\label{strong generating set h2} 
 \{\omega^{1,1}_{0,0},\omega^{1,2}_{0,0}, \omega^{1,2}_{0,1},\omega^{1,2}_{0,2},\omega^{2,2}_{0,0},\omega^{2,2}_{0,2},\},
\end{equation}
and is of type $\mathcal{W}(2^{3},3,4^{2})$.
\item For $n\geq3,$ $\mathcal{H}{(n)}^{\mathbb{Z}_{2}}$ has a minimal strong generating set
\begin{equation}\label{strong generating set hn}
 \{\omega_{0,0}^{i,i}|\ i=1,\dots ,n\} \bigcup  \{\omega^{i,j}_{0,0}, \omega^{i,j}_{0,1}|\ 1\leq i< j\leq n\}  \bigcup \{\omega_{0,2}^{1,1}\},
 \end{equation}
 and is of type $\mathcal{W}(2^{n+\binom{n}{2}},3^{\binom{n}{2}},4).$
 \end{enumerate}
 \end{thm}

\begin{proof}
First, we consider the case $n = 2$. By replacing $\omega_{0,2m}$ with $\omega^{i,i}_{0,2m}$ in \eqref{PH14} and \eqref{n=1:higherdecoup} for $i = 1,2$, we obtain decoupling relations 
$$\omega^{i,i}_{0,2m} = P_{2m}(\omega^{i,i}_{0,0}, \omega^{i,i}_{0,2})$$ for $i = 1,2$ and all $m\geq 2$. Then by Lemma \ref{strong generators for H(n)}, $\mathcal{H}{(2)}^{\mathbb{Z}_{2}}$ has a strong generating set 
$$\{\omega^{1,1}_{0,0},\omega^{1,1}_{0,2},\omega^{2,2}_{0,0},\omega^{2,2}_{0,2}\} \bigcup \{\omega^{1,2}_{0,m},m\geq 0\}.$$ 

Next, we have the following relation at weight $5$ among the generators of $\text{gr}(\mathcal{H}(2)^{\mathbb{Z}_{2}})$
\begin{equation}\label{deal1}
q^{1,2}_{0,0} q ^{2,2}_{0,1} -q^{1,2}_{0,1} q^{2,2}_{0,0} =0.
\end{equation}
 The corresponding element $: \omega^{1,2}_{0,0}  \omega ^{2,2}_{0,1} :-: \omega^{1,2}_{0,1} \omega ^{2,2}_{0,0}  :$ in $(\mathcal{H}(2)^{\mathbb{Z}_{2}})_{(2)}$ has a correction of the form  
\begin{equation}\label{eq3}
: \omega^{1,2}_{0,0}  \omega ^{2,2}_{0,1} :-: \omega^{1,2}_{0,1} \omega ^{2,2}_{0,0}  : \ =-\frac{1}{2}  \omega ^{1,2}_{0,3}+{2} \partial \omega^{1,2}_{0,2}-\frac{5}{2} \partial ^{2} \omega^{1,2}_{0,1} + \partial ^{3} \omega^{1,2}_{0,0}.
\end{equation}
This can be rewritten as follows:
 \begin{equation}\label{eqQ4}
 \omega ^{1,2}_{0,3}= Q_3( \omega_{0,0}^{2,2},\omega^{1,2}_{0,0},\omega^{1,2}_{0,1},\omega^{1,2}_{0,2}),
 \end{equation} where $Q( \omega_{0,0}^{2,2},\omega^{1,2}_{0,0},\omega^{1,2}_{0,1},\omega^{1,2}_{0,2})$ is a normally ordered polynomial in $\omega^{2,2}_{0,0}, \omega^{1,2}_{0,0},\omega^{1,2}_{0,1},\omega^{1,2}_{0,2}$, and their derivatives
 
 Next, by applying the operator $\omega^{2,2}_{0,2}\circ_{1}$ repeatedly, we can get decoupling relations
 \begin{equation}\label{n=2:higherdecoup} \omega^{1,2}_{0,m} = Q_m(\omega_{0,0}^{2,2},\omega_{0,2}^{2,2},\omega^{1,2}_{0,0},\omega^{1,2}_{0,1},\omega^{1,2}_{0,2}),
 \end{equation} for all $m>3$. This follows from the calculation
\begin{equation}
\omega^{2,2}_{0,1}\circ_{1} \omega^{1,2}_{0,k}= -\omega^{1,2}_{0,k+1}.
\end{equation}
Therefore $\omega^{1,2}_{0,m}$ for $m \geq 3$ are not necessary. Finally, we have the relation
 \begin{equation} :\omega^{1,2}_{0,0}\omega^{1,2}_{0,0}:\ =\frac{1}{2}\omega^{1,1}_{0,2}+\frac{1}{2}\omega^{2,2}_{0,2}+:\omega^{1,1}_{0,0}\omega^{2,2}_{0,0}:. \end{equation}
This shows that $\omega^{2,2}_{0,2}$ is unnecessary, hence the set \eqref{strong generating set h2} suffices to strongly generate $\mathcal{H}(2)^{\mathbb{Z}_2}$. The fact that this set is a minimal strong generating set is clear since there are no relations of weight less than $5$ of the form \eqref{ideal'}.

Finally, we consider the case $n\geq 3$. As above, for $i = 1,\dots, n$ and $m \geq 2$ we have relations
 $$\omega^{i,i}_{0,2m} = P_m(\omega^{i,i}_{0,0}, \omega^{i,i}_{0,2}).$$ So $\omega^{i,i}_{0,2m}$ can be eliminated for all $m\geq 2$.
 
For $n\geq 3$, we can do a bit better than the relations \eqref{eqQ4} and \eqref{n=2:higherdecoup}, since we have the following relations at weight $4$ in $\text{gr}(\mathcal{H}(n)^{\mathbb{Z}_2})$ for all $i<j<k$,
 \begin{equation}\label{deal2}
q^{i,j}_{0,0} q ^{j,k}_{0,0} -q^{i,k}_{0,0} q^{j,j}_{0,0} =0.
\end{equation}
The corresponding element $: \omega_{0,0}^{i,j} \omega _{0,0}^{j,k} :-: \omega_{0,0}^{i,k} \omega _{0,0}^{j,j} :$ lies in $(\mathcal{H}(n)^{\mathbb{Z}_{2}})_{(2)},$ and has a correction of the form 
\begin{equation}\label{3copies} : \omega_{0,0}^{i,j} \omega _{0,0}^{j,k} :-: \omega_{0,0}^{i,k} \omega _{0,0}^{j,j} : \ =\frac{1}{2}\omega^{i,k}_{0,2}-\partial \omega^{i,k}_{0,1}+\frac{1}{2}\partial^{2}\omega^{i,k}_{0,0}.\end{equation}
 
We can clearly rewrite \eqref{3copies} in the form
 \begin{equation}\label{equationQ'4}
\omega^{i,k}_{0,2} = T_{2}( \omega_{0,0}^{j,j},\omega_{0,0}^{i,j} ,\omega _{0,0}^{j,k} ,\omega^{i,k}_{0,0},\omega^{i,k}_{0,1}),
 \end{equation}
 where $T_{2}( \omega_{0,0}^{j,j},\omega_{0,0}^{i,j} ,\omega _{0,0}^{j,k} ,\omega^{i,k}_{0,0},\omega^{i,k}_{0,1})$ is a normally ordered polynomial in $ \omega_{0,0}^{j,j},\omega_{0,0}^{i,j} ,\omega _{0,0}^{j,k} ,\omega^{i,k}_{0,0},\omega^{i,k}_{0,1}$, and their derivatives.
 
As above, by applying the operator $\omega^{k,k}_{0,2} \circ_1$ repeatedly, we can construct relations
\begin{equation} \omega^{i,k}_{0,m} = T_m(\omega_{0,0}^{j,j},\omega_{0,0}^{i,j} ,\omega _{0,0}^{j,k},\omega _{0,1}^{j,k} ,\omega^{i,k}_{0,0},\omega^{i,k}_{0,1})\end{equation} for all $m\geq 2$. This shows that $\omega^{i,k}_{0,m}$ can be eliminated for all $m\geq 2$.

 Finally, for all $j$ with $1<j \leq n$, we have the relation
\begin{equation} :\omega^{1,j}_{0,0}\omega^{1,j}_{0,0}:\ =\frac{1}{2}\omega^{1,1}_{0,2}+\frac{1}{2}\omega^{j,j}_{0,2}+:\omega^{1,1}_{0,0}\omega^{j,j}_{0,0}:. \end{equation}
 This shows that $\omega^{j,j}_{0,2}$ can be eliminated for $1< j\leq n$. It follows that \eqref{strong generating set hn} strongly generates $\mathcal{H}(n)^{\mathbb{Z}_2}$ for $n\geq 3$. The fact that it is a minimal strong generating set is again clear since there are no more relations of weight less than $5$ of the form \eqref{ideal'}. \end{proof}

Notice that $\omega_{0,0}^{i,i},\omega_{0,2}^{i,i},\omega^{i,j}_{0,0}, \omega^{i,j}_{0,1}, \omega^{i,j}_{0,2}$ are not primary fields with respect to the Virasoro field $$L(z)=\frac{1}{2}\sum_{i= 1}^{n} :a^{i}(z)a^{i}(z): = \sum_{i= 1}^{n} \omega^{i,i}_{0,0}.$$ It is easy to correct them to be a primary ones by adding a normally ordered polynomial in the previous set and their derivatives. By a computer calculation, we obtain the following primary fields:
\begin{equation}\label{10}
\begin{aligned}
C^{k}&=\frac{1}{2}(\omega^{1,1}_{0,0}-\omega^{k,k}_{0,0}),\enspace where\enspace k=2,\dots ,n.\\
C^{i,i}_{0,2}&=\omega^{i,i}_{0,2}-\frac{2}{9}:\omega^{i,i}_{0,0}\omega^{i,i}_{0,0}:-\frac{1}{6}\partial^{2}\omega^{i,i}_{0,0},\enspace where\enspace i = 1,\dots ,n,\\
C^{i,j}_{0,0}&=\omega^{i,j}_{0,0},\\
C^{i,j}_{0,1}&=\omega^{i,j}_{0,1}-\frac{1}{2}\partial \omega^{i,j}_{0,0},\\
C^{i,j}_{0,2}&=\omega^{i,j}_{0,2}-\frac{4}{9}:\omega^{i,j}_{0,0}\omega^{j,j}_{0,0}:+\frac{5}{9}\partial^{2}\omega^{i,j}_{0,0}-\frac{13}{9}\partial\omega^{i,j}_{0,1}.
\end{aligned}
\end{equation}

\section{The Cartan involution and its extension to $V^k(\lie{g})$}

Let us consider a simple Lie algebra $\lie{g}$ as above with $l = \text{rank}(\lie{g})$ and $m$ the number of positive roots. With respect to a choice of base for the root system $\Phi$, we have the triangular decomposition
$$\lie{g}=\lie{h}\oplus \lie{n}_{+} \oplus \lie{n}_{-},$$ where $\lie{h}$ is the Cartan subalgebra with basis $h_r$, $r = 1,\dots, l$, and $\lie{n}_+$ has basis $x_{\beta_i}$ for $i = 1,\dots, m$, and $\lie{n}_-$ has basis $y_{\beta_i}$ for $i = 1,\dots, m$. The Cartan involution $\theta$ of $\lie{g}$ is defined as follows:
 \begin{equation*}
 \theta(x_{\beta_{i}})=-y_{\beta_{i}}, \enspace \enspace \enspace  \theta(y_{\beta_{i}})=-x_{\beta_{i}}, \enspace \enspace \enspace \theta(h_{r})=-h_{r}.
 \end{equation*}

Since $\theta$ preserves the Lie bracket as well as the normalized Killing form, it extends to an automorphism of the vertex algebra $V^k(\lie{g})$ given by the same formula, where $h_r, x_{\beta_i}, y_{\beta_i}$ are now considered as the generating fields for $V^k(\lie{g})$.

To replace the generators $h_r, x_{\beta_i}, y_{\beta_i}$ of $V^k(\lie{g})$ with a set of eigenvectors for $\theta,$ it is suitable to apply a linear change of variables as follows:
 \begin{equation*}
 E_{\beta_{i}}= x_{\beta_{i}}+y_{\beta_{i}}, \enspace \enspace \enspace F_{\beta_{i}}=x_{\beta_{i}}-y_{\beta_{i}} \enspace \enspace \enspace h_{r}.
 \end{equation*}
 The action of $\theta$ on the new generators will be as follows:
 \begin{equation*}
 \theta(E_{\beta_{i}})=-E_{\beta_{i}},\enspace \enspace \enspace  \theta(F_{\beta_{i}})=F_{\beta_{i}}, \enspace \enspace \enspace \theta(h_{r})=-h_{r}.
 \end{equation*}
There is a PBW basis consisting of normally ordered monomials of the new generators and their derivatives since the new generators are related to the old ones by a linear change of variables.

Note that the fields $F_{\beta_i}$ lie in the orbifold $V^k(\lie{g})^{\mathbb{Z}_2}$. Define additional generators of $V^{k}(\lie{g})^{\mathbb{Z}_2}$ as follows: 
\begin{equation*}
\begin{aligned}
Q_{a,b}^{\beta_{i},\beta_{j}}&= \ :\partial^{a} E_{\beta_{i}}(z) \partial^{b} E_{\beta_{j}}(z):,\\ 
Q_{a,b}^{h_{r},\beta_{i}}&= \ : \partial^{a} h_{r}(z)\partial^{b} E_{\beta_{i}}(z):,\\ 
Q_{a,b}^{h_{r},h_{s}}&= \ :\partial^{a} h_{r}(z) \partial^{b} h_{s}(z):,
\end{aligned}
\end{equation*}
 which each have weight $a+b+2.$

A special case when $\lie{g} = \mathfrak{sl}_2$, we have only one positive root $\beta.$ An immediate consequence that there is one element $F_{\beta}$, one element $E_{\beta}$, and one basis vector $h$ for $\lie{h}$. The above elements can be given as follows: \begin{equation*}
\begin{aligned}
Q_{a,b}^{\beta,\beta}&= \ :\partial^{a} E_{\beta}(z) \partial^{b} E_{\beta}(z):,\\ 
Q_{a,b}^{h,\beta}&= \ : \partial^{a} h(z)\partial^{b} E_{\beta}(z):,\\ 
Q_{a,b}^{h,h}&= \ :\partial^{a} h(z) \partial^{b} h(z):.
\end{aligned}
\end{equation*}

\section{The $\mathbb{Z}_2$-orbifold of $V^k(\lie{g})$}

At this point, we are ready to state our main theorem which describes $V^{k}(\lie{g})^{\mathbb{Z}_{2}}$ for generic values of $k$.
\begin{thm}\label{ch5:mainresult}

Let $\lie{g}$ be a simple, finite-dimensional Lie algebra, and let $l = \text{rank}(\lie{g})$, $m$ the number of positive roots, and set $d = m+l$.
\begin{enumerate}

\item For $\lie{g} \neq \mathfrak{sl}_2$, and $k$ generic, $V^k(\lie{g})^{\mathbb{Z}_{2}}$ has a minimal strong generating set
\begin{equation*}
\{F_{\beta_{i}},Q_{0,0}^{\beta_{a},\beta_{b}},Q_{0,2}^{\beta_{1},\beta_{1}},Q_{0,0}^{h_{r},h_{s}},Q_{0,0}^{h_{t},\beta_{u}},Q_{0,1}^{h_{t},\beta_{u}}\} ,
\end{equation*}
for $1\leq i \leq m$, $1\leq a \leq b \leq m$, $1\leq r \leq s\leq l$, $1\leq t \leq l$, and $1\leq u \leq m$. In particular, $V^k(\lie{g})^{\mathbb{Z}_{2}}$ is of type $\mathcal{W}(1^{m},2^{d+ \binom{d}{2}},3^{ \binom{d}{2}},4)$.

\item For $\lie{g}=\mathfrak{sl}_{2}$ and $k$ generic, $V^{k}(\lie{g})^{\mathbb{Z}_{2}}$ has a minimal strong generating set
\begin{equation*}
\{F,Q^{\beta,\beta}_{0,0},Q^{h,h}_{0,0},Q^{h,h}_{0,2},Q^{\beta,h}_{0,0},Q^{\beta,h}_{0,1},Q^{\beta,h}_{0,2}\} ,
\end{equation*}
 and in particular, is of type $\mathcal{W}(1,2^{3},3,4^{2}).$
 \end{enumerate}
\end{thm}

\begin{proof}
Let $n = 2m+l = \text{dim}(\lie{g})$. By the deformation argument of \cite{LII}, we have $$\lim_{k \to \infty}V^{k}(\lie{g}) \cong \mathcal{H}(n).$$ Here $\mathcal{H}(n)$ is the rank $n$ Heisenberg algebra with generators $F_{\beta_i}, E_{\beta_i}$ and $h_r$. We shall use the same symbols for the limits of these fields when no confusion may arise. Moreover, $\mathbb{Z}_2$ acts trivially on the rank $m$ Heisenberg subalgebra generated by $\{F_{\beta_i}|\ i = 1,\dots, m$, and it acts by $-1$ on the rank $d = m+ l$ Heisenberg algebra with generators $\{E_{\beta_i}, h_r|\ i = 1,\dots, m,\ \ r = 1,\dots, l\}$.
We get
$$\lim_{k \to \infty}V^{k}(\lie{g})^{\mathbb{Z}_2} \cong \mathcal{H}(m) \otimes \big(\mathcal{H}(d)^{\mathbb{Z}_2}\big).$$ In the limit $k\rightarrow \infty$, the fields $F_{\beta_i}$ are the generators of $\mathcal{H}(m)$, and the remaining quadratic fields are precisely the generators for $\mathcal{H}(d)^{\mathbb{Z}_2}$. The claim in both cases then follows by Theorem \ref{strong generators for Hn}.
\end{proof}

For the reader's convenience, the following table shows the dimension, rank $l$, and the number of positive roots $m$, for each simple Lie algebra in the Cartan-Killing classification.

\begin{center}
 \begin{tabular}{||c c c c||} 
 \hline
 Lie algebra & Dimension & Rank $l $& The number of positive roots $m$ \\ [0.5ex] 
 \hline\hline
 $\mathfrak{sl}_{n+1}$ & $(n+1)^{2}-1$ & $n$ & $\frac{(n^{2}+n)}{2}$ \\ 
 \hline
 $\mathfrak{so}_{2n+1}$ &$\frac{2n(2n+1)}{2}$ &  $n$ & $n^{2}$ \\
 \hline
 $\mathfrak{sp}_{2n}$ & $n(2n+1)$ &  $n$ & $n^{2}$ \\
 \hline
 $\mathfrak{so}_{2n}$ & $\frac{2n(2n-1)}{2}$ &  $n$ & $n^{2}-n$ \\
 \hline
 $G_{2}$ & $14$ & 2 & $6$\\
  \hline
 $F_{4}$ & $52$ & 4 & $24$\\
 \hline
$E_{6}$ & $78$ & 6 & $36$\\
 \hline
 $E_{7}$ & $133$ & 7 & $63$\\
  \hline
 $E_{8}$ & $248$ & 8 & $120$\\ [1ex] 
  \hline
\end{tabular}
\end{center}

\section{The nongeneric set for $V^{k}(\mathfrak{sl}_{2})^{\mathbb{Z}_{2}}$}

\hphantom{X} In this section we work in the usual root basis for $\mathfrak{sl}_{2},$ so we shall change our notation here. Let $\{x,y,h \}$ be an ordered basis of $\mathfrak{sl}_{2}$ satisfying the following commutation relations:
$$[x,y]=h,\enspace [h,x]=2x, \enspace [h,y]=-2y.$$ 
Let $X^{x}, X^{y},X^{h}$ be generating fields for $V_{k}(\mathfrak{sl}_{2}),$ where each of conformal weight $1,$ and satisfies the OPE relations
\begin{equation}\label{opesl2}
X^{x}(z) X^{y}(w) \sim k (z-w)^{-2}+X^{h}(w)(z-w)^{-1},
\end{equation}

\begin{equation}
X^{h}(z)X^{x}(w) \sim 2 X^{x}(w) (z-w)^{-1},
\end{equation}

\begin{equation}
X^{h}(z)X^{y}(w)\sim -2 X^{y}(w) (z-w)^{-1},
\end{equation}

\begin{equation}\label{opesl2'}
X^{h}(z)X^{h}(w)\sim 2k (z-w)^{-2}.
\end{equation}
The affine vertex algebra $V_{k}(\mathfrak{sl}_{2})$ is freely generated by the even generators $X^{x}, X^{y},X^{h}$ and in particular it has a PBW basis as follows
\begin{equation}\label{basissl2}
:\partial^{k^{1}_{1}} X^{x}\cdots \partial^{k^{1}_{s_{1}}} X^{x}\partial^{k^{2}_{1}} X^{y}\cdots \partial^{k^{2}_{s_{2}}} X^{y} \partial^{k^{3}_{1}} X^{h}\cdots \partial^{k^{3}_{s_{3}}} X^{h}:, 
\end{equation}
$ s_{i} \geq0, \enspace k^{i}_{1}\geq \cdots \geq k^{i}_{s_{i}} \geq 0,\enspace for \enspace i=1,2,3.$
\newline

The action of $\theta$ is given by 
\begin{equation*}
\theta(X^{x})= -X^{y},\enspace \enspace\enspace \enspace \theta(X^{y})= -X^{x},\enspace \enspace\enspace \enspace \theta(X^{h})=-X^{h}.
\end{equation*}
Changing the basis to the basis of eigenvectors yields:
\begin{equation}\label{newbasis}
G=X^{x}+ X^{y},\enspace \enspace \enspace \enspace F=X^{x}-X^{y}, \enspace \enspace \enspace \enspace H=X^{h}.
\end{equation}
The nontrivial involution $\theta$ acts on the new generators as follows:
\begin{equation*}
\theta(G)=- G,\enspace \enspace\enspace \enspace \theta(F)= F,\enspace \enspace\enspace \enspace \theta(H)=-H.
\end{equation*}

Define 
\begin{equation*}
\begin{aligned}
Q_{i,j}&= \ :\partial^{i} G(z) \partial^{j} G(z):\ \in (V^{k}(\mathfrak{sl}_{2})^{\mathbb{Z}_{2}})_{(2)},\\
U_{i,j}&= \ :\partial^{i} H(z) \partial^{j}H(z):\  \in (V^{k}(\mathfrak{sl}_{2})^{\mathbb{Z}_{2}})_{(2)},\\
V_{i,j}&= \ :\partial^{i} H(z) \partial^{j}G(z):\ \in (V^{k}(\mathfrak{sl}_{2})^{\mathbb{Z}_{2}})_{(2)}
\end{aligned}
\end{equation*} 
as new generators for $V^{k}(\mathfrak{sl}_{2}),$ where each have weight $i+j+2.$

By Theorem \ref{ch5:mainresult}, the strong generators for $V^k(\mathfrak{sl}_2)$ are $\{F,Q_{0,0},U_{0,0},U_{0,2},V_{0,0},V_{0,1},V_{0,2}\}$. These fields close under OPE in the sense that for any $\alpha_1, \alpha_2$ in the above set, each term in the OPE of $\alpha_1(z) \alpha_2(w)$ can be expressed as a linear combination of normally ordered monomials in these generators. The coefficients of these monomials are called the {\it structure constants} of the OPE algebra, and they are all rational functions of $k$. The set of nongeneric values of $k$ where the strong finite generating set for $V^{k}(\mathfrak{sl}_{2})^{\mathbb{Z}_{2}}$ does not work can be determined here. By Theorem 5.3 of \cite{CL}, the only nongeneric values of $k$ are the poles of these structure constants. Clearly there are at most finitely many such poles. Using K. Thielemans' Mathematica package \cite{T}, the full OPE algebra among these generators was calculated, and we find that the poles of the structure constants lie in the set $$\{0,\frac{16}{51},\frac{16}{9},-\frac{32}{3}\}.$$ It follows that all other values of $k$ are generic.

An immediate consequence is the following:
\begin{cor}
 For $k\neq0,\frac{16}{51},\frac{16}{9},-\frac{32}{3},$ $V^{k}(\lie{sl}_{2})^{\mathbb{Z}_{2}}$ is of type $\mathcal{W}(1,2^3, 3, 4^2)$.
\end{cor}

\section{Acknowledgements}

I would like to express my deep appreciation and gratitude to Prof. Linshaw for his immense knowledge, useful discussions, and valuable suggestions throughout this work.

\end{document}